\documentclass[12pt,a4paper,reqno]{amsart}
\usepackage{amscd,amsmath,amsthm,amssymb}

\def\NZQ{\mathbb}               

\def\ZZ{{\NZQ Z}}
\def\RR{{\NZQ R}}

\def\Pc{{\mathcal P}}
\def\Qc{{\mathcal Q}}

\def\Bc{{\mathcal B}}
\def\Cc{{\mathcal C}}

\def\eb{{\mathbf e}}
\def\wb{{\mathbf w}}

\def\opn#1#2{\def#1{\operatorname{#2}}} 
\opn\chara{char} \opn\length{\ell} \opn\pd{pd} \opn\rk{rk}
\opn\projdim{proj\,dim} \opn\injdim{inj\,dim} \opn\rank{rank}
\opn\depth{depth} \opn\grade{grade} \opn\height{height}
\opn\embdim{emb\,dim} \opn\codim{codim}
\opn\Cl{Cl}

\opn\Tr{Tr} \opn\bigrank{big\,rank}
\opn\superheight{superheight}\opn\lcm{lcm}
\opn\trdeg{tr\,deg}
\opn\rdeg{rdeg}
	\opn\reg{reg} \opn\lreg{lreg} \opn\ini{in} \opn\lpd{lpd}
	\opn\size{size} \opn\sdepth{sdepth}
	\opn\link{link}\opn\fdepth{fdepth}\opn\lex{lex}
	\opn\tr{tr}
	\opn\type{type}
	\opn\gap{gap}
	\opn\arithdeg{arith-deg}
	\opn\revlex{revlex}
	\opn\div{div} \opn\Div{Div} \opn\cl{cl} \opn\Cl{Cl}
	\opn\Spec{Spec} \opn\Supp{Supp} \opn\supp{supp} \opn\Sing{Sing}
	\opn\Ass{Ass} \opn\Min{Min}\opn\Mon{Mon}
	\opn\Ann{Ann} \opn\Rad{Rad} \opn\Soc{Soc}
	\opn\Im{Im} \opn\Ker{Ker} \opn\Coker{Coker} \opn\Am{Am}
	\opn\Hom{Hom} \opn\Tor{Tor} \opn\Ext{Ext} \opn\End{End}
	\opn\Aut{Aut} \opn\id{id}
	
	\opn\nat{nat}
	\opn\pff{pf}
	\opn\Pf{Pf} \opn\GL{GL} \opn\SL{SL} \opn\mod{mod} \opn\ord{ord}
	\opn\Gin{Gin} \opn\Hilb{Hilb}\opn\sort{sort}
	\opn\PF{PF}\opn\Ap{Ap}
	\opn\mult{mult}
	\opn\bight{bight}
	\opn\div{div}
	\opn\Div{Div}
	\opn\aff{aff}
	\opn\relint{relint} \opn\st{st}
	\opn\lk{lk} \opn\cn{cn} \opn\core{core} \opn\vol{vol}  \opn\inp{inp} \opn\nilpot{nilpot}
	\opn\link{link} \opn\star{star}\opn\lex{lex}\opn\set{set}
	\opn\width{wd}
	\opn\Fr{F}
	\opn\QF{QF}
	\opn\G{G}
	\opn\type{type}\opn\res{res}
	\opn\conv{conv}
	\opn\Int{Int}
	\opn\Deg{Deg}
	\opn\Sym{Sym}
	\opn\Con{Con}
	\opn\gr{gr}
	
	\def\pot#1#2{#1[\kern-0.28ex[#2]\kern-0.28ex]}

	\opn\dirlim{\underrightarrow{\lim}}
	\opn\inivlim{\underleftarrow{\lim}}
	\def\Implies{\ifmmode\Longrightarrow \else
		\unskip${}\Longrightarrow{}$\ignorespaces\fi}
	\def\implies{\ifmmode\Rightarrow \else
		\unskip${}\Rightarrow{}$\ignorespaces\fi}
	\def\iff{\ifmmode\Longleftrightarrow \else
		\unskip${}\Longleftrightarrow{}$\ignorespaces\fi}

	\let\:=\colon
	\newtheorem{Theorem}{Theorem}[section]
	\newtheorem{Lemma}[Theorem]{Lemma}
	\newtheorem{Corollary}[Theorem]{Corollary}
	
	\newtheorem{Conjecture}[Theorem]{Conjecture}
	
	\theoremstyle{definition}

	\newtheorem{Example}[Theorem]{Example}
	\let\epsilon\varepsilon
	\let\kappa=\varkappa
	\opn\dis{dis}
	\def\pnt{{\raise0.5mm\hbox{\large\bf.}}}
	
	\opn\Lex{Lex}

\begin{document}

\title{Lattice polytopes with the minimal volume}
\author{Ginji Hamano, Ichiro Sainose and Takayuki Hibi}
\address{Ginji Hamano, School of Science and Engineering, Tokyo Denki University, 
Saitama 350-0394, Japan}
\email{18hz002@ms.dendai.ac.jp}
\address{Ichiro Sainose, 
Hiroshima Municipal Motomachi Senior High School, 
Hiroshima, 730-0005, Japan}
\email{sainose23517257@nifty.com}
\address{Takayuki Hibi, Department of Pure and Applied Mathematics, Graduate School of Information Science and Technology, Osaka University, Suita, Osaka 565-0871, Japan}
\email{hibi@math.sci.osaka-u.ac.jp}
\dedicatory{}
\begin{abstract}
Let $\Pc \subset \RR^d$ be a lattice polytope of dimension $d$.  Let $b(\Pc)$ denote the number of lattice points belonging to the boundary of $\Pc$ and $c(\Pc)$ that to the interior of $\Pc$.  It follows from the lower bound theorem of Ehrhart polynomials that, when $c > 0$, 
\[
{\rm vol}(\Pc) \geq (d \cdot c(\Pc) + (d-1) \cdot b(\Pc) - d^2 + 2)/d!,
\]
where ${\rm vol}(\Pc)$ is the (Lebesgue) volume of $\Pc$.  Pick's formula guarantees that, when $d = 2$, the above inequality is an equality.  
In the present paper several classes of lattice polytopes for which the equality here holds will be presented. 
\end{abstract}	
\maketitle
\thispagestyle{empty}

\section*{Introduction}
Recall that a {\em lattice point} of $\RR^d$ is a point belonging to $\ZZ^d$.  Let $\Pc \subset \RR^d$ be a {\em lattice polytope} of dimension $d$.  In other words, $\Pc$ is a convex polytope of dimension $d$ each of whose vertices is a lattice point.  Let $b(\Pc)$ denote the number of lattice points belonging to the boundary of $\Pc$ and $c(\Pc)$ the number of lattice points belonging to the interior of $\Pc$.  The lower bound theorem of Ehrhart polynomials \cite{hibi} guarantees that, when $c > 0$, 
\begin{eqnarray}
\label{formula}
{\rm vol}(\Pc) \geq (d \cdot c(\Pc) + (d-1) \cdot b(\Pc) - d^2 + 2)/d!,
\end{eqnarray}
where ${\rm vol}(\Pc)$ is the (Lebesgue) volume of $\Pc$.  Even though the argument done in \cite{hibi} is rather complicated with deep techniques on polytopes, a short and elementary proof of (\ref{formula}) is presented in \cite{SHEH}.  Pick's formula says that the inequality (\ref{formula}) is an equality when $d = 2$.  

A lattice polytope $\Pc \subset \RR^d$ of dimension $d$ is called {\em Castelnuovo} \cite{kawaguchi} if the equality holds in (\ref{formula}).  Only a few classes of Castelnuovo polytopes are known.  A triplet of integers $(b, c, d)$ with $b \geq c+d+1, \, c \geq 1$ and $d \geq 3$ is called {\em Castelnuovo} if there exists a Castelnuovo polytope $\Pc \subset \RR^d$ of dimension $d$ with $b = b(\Pc)$ and $c = c(\Pc)$.  One of the reasonable problems on Castelnuovo polytopes is to find all of the Castelnuovo triplets.  In fact, to find the complete list of the Castelnuovo triplets establishes the foundation to classify the Castelnuovo polytopes.  

When $d = 2$, Scott \cite{Scott} finds all of the Castelnuovo triplets.  In fact, the Castelnuovo triplets with $d = 2$ are (i) $(b,1,2)$ with $3 \leq b \leq 9$ and (ii) $(b,c,2)$ with $c \geq 2$ and $3 \leq b \leq 2c + 6$.
  
On the other hand, it is known \cite{hibi_tsuchiya} that the triplet of integers $(d+1, c, d)$, where $c \geq 1$ and $d \geq 3$, are Castelnuovo pairs.  In the present paper, several classes of Castelnuovo polytopes and a partial list of the Castelnuovo triplets will be presented.     

\section{Castelnuovo polytopes of bipyramid type}
Let $\eb_1, \ldots, \eb_d$ denote the canonical unit coordinate vectors of $\RR^d$.  Fix integers $c \geq 1, \, d \geq 3$ and $0 \leq n < cd$.  We introduce the lattice polytope $$\Pc_{c,d}(n) \subset \RR^d,$$ called {\em of bipyramid type}, of dimension $d$ whose vertices are 
\[
{\bf 0}, \, \, \eb_1, \ldots, \eb_{d-1}, \, \, \eb_d + n \sum_{i=1}^{d} \eb_i, \, \, \eb_d + cd \sum_{i=1}^{d} \eb_i,
\]
where ${\bf 0}$ is the origin of $\RR^d$.  

\begin{Example}
The vertices of $\Pc_{3,4}(5) \subset \RR^4$ are 
\[
(0,0,0,0), (1,0,0,0), (0,1,0,0), (0,0,1,0), (5,5,5,6), (12,12,12,13).
\]
One has $b(\Pc_{3,4}(5)) = 12, \, c(\Pc_{3,4}(5)) = 3$ and ${\rm vol}(\Pc_{3,4}(5)) = 34/4!$.  Hence $\Pc_{3,4}(5)$ is Castelnuovo.
\end{Example}

\begin{Lemma}
\label{interior}
Each of the lattice points 
\[
\sum_{i=1}^{d} \eb_i, \, \, 2 \sum_{i=1}^{d} \eb_i, \ldots, c \sum_{i=1}^{d} \eb_i
\]
belongs to the interior of $\Pc_{c,d}(n)$.
\end{Lemma}

\begin{proof}
Let $1 \leq j \leq c$.  Then 
\[
j \sum_{i=1}^{d} \eb_i = \frac{(c - j)d + 1}{\,cd+1\,} {\bf 0} + \sum_{i=1}^{d-1} \frac{j}{\,cd+1\,} \eb_i + \frac{j}{\,cd+1\,} (\eb_d + cd \sum_{i=1}^{d} \eb_i)
\]
belongs to the interior of a simplex of dimension $d$ which is contained in $\Pc_{c,d}(n)$.
\end{proof}

\begin{Lemma}
\label{boundary}
Each of the lattice points
\[
{\bf 0}, \, \, \eb_1, \ldots, \eb_{d-1}, \, \, \eb_d + n \sum_{i=1}^{d} \eb_i, \, \, \eb_d + (n+1) \sum_{i=1}^{d} \eb_i, \ldots,  \eb_d + cd \sum_{i=1}^{d} \eb_i.  
\]
belongs to the boundary of $\Pc_{c,d}(n)$. 
\end{Lemma}

\begin{proof}
Let $n < j < c$.  Then 
\[
\eb_d + j \sum_{i=1}^{d} \eb_i = 
\frac{cd-j}{\,cd-n\,}(\eb_d + n \sum_{i=1}^{d} \eb_i) + \frac{j-n}{\,cd-n\,}(\eb_d + cd \sum_{i=1}^{d} \eb_i)
\]
belongs to an edge of $\Pc_{c,d}(n)$. 
\end{proof}

Now, in order to compute the volume of $\Pc_{c,d}(n)$, a triangulation of $\Pc_{c,d}(n)$ is studied.  Let $\sigma \subset \RR^d$ denote the lattice simplex of dimension $d$  with the vertices
\[
{\bf 0}, \, \, \eb_1, \ldots, \eb_{d-1}, \, \, \eb_d + cd \sum_{i=1}^{d} \eb_i.
\] 
Furthermore, for each $1 \leq i \leq d - 1$, we introduce the lattice simplex $\sigma(i) \subset \RR^d$ of dimension $d$ whose vertices are
\[
{\bf 0}, \, \, \eb_1, \ldots, \eb_{i-1}, \eb_{i+1}, \ldots, \eb_{d-1}, \, \, \eb_d + n \sum_{i=1}^{d} \eb_i, \, \, \eb_d + cd \sum_{i=1}^{d} \eb_i.
\]

A standard technique \cite[Chapter 4]{gtm279} for triangulations guarantees the following   

\begin{Lemma}
\label{triangulation}
Let $\Sigma$ denote the set of simplices consisting of $\sigma, \sigma(1), \ldots, \sigma(d-1)$ together with their faces.  Then $\Sigma$ is a triangulation of $\Pc_{c,d}(n)$.
\end{Lemma}

It is known that the volume of a simplex is computed by using the determinant.  In particular, the volume of $\sigma \subset \RR^d$ is $(cd + 1)/d!$ and the volume of each $\sigma(i) \subset \RR^d$ is $(cd - n)/d!$.  Corollary \ref{volume} follows from Lemma \ref{triangulation}.

\begin{Corollary}
\label{volume}
The volume of $\Pc_{c,d}(n)$ is
\[
{\rm vol}(\Pc_{c,d}(n)) = ((d-1)(cd - n) + (cd + 1))/d!. 
\]
\end{Corollary}

Now, we come to the highlight of the present section.

\begin{Theorem}
\label{highlight}
The lattice polytope $\Pc_{c,d}(n)$ is Castelnuovo with 
\[
b(\Pc_{c,d}(n)) = cd - n + (d + 1), \, \, \, \, \, c(\Pc_{c,d}(n)) = c.
\]
\end{Theorem}

\begin{proof}
Lemmas \ref{interior} and \ref{boundary} say that  
$c(\Pc_{c,d}(n)) \geq c$
and
$b(\Pc_{c,d}(n)) \geq cd - n +(d + 1)$.
By virtue of (\ref{formula}) together with Corollary \ref{volume}, it follows that
\begin{eqnarray*}
{\rm vol}(\Pc_{c,d}(n)) 
& \geq & (d \cdot c(\Pc_{c,d}(n)) + (d-1) \cdot b(\Pc_{c,d}(n)) - d^2 + 2)/d! \\
& \geq & (cd + (d-1)(cd - n +(d + 1)) - d^2 + 2)/d! \\
& = & ((d-1)(cd - n) + cd + (d-1)(d+1) - d^2 + 2)/d! \\
& = & ((d-1)(cd - n) + (cd + 1))/d! \\
& = & {\rm vol}(\Pc_{c,d}(n)). 
\end{eqnarray*}
Hence $c(\Pc_{c,d}(n)) = c$
and
$b(\Pc_{c,d}(n)) = cd - n +(d + 1)$.  Furthermore, the equality holds in (\ref{formula}).  Thus $\Pc_{c,d}(n)$ is Castelnuovo, as desired.  
\end{proof}

\section{Castelnuovo polytopes of prism type}
Let $d \geq 3$ be an integer.  Let, as before, $\eb_1, \ldots, \eb_d$ denote the canonical unit coordinate vectors of $\RR^d$ and $\wb = - \eb_1 - \cdots -\eb_{d-1}$.  
Let $\Qc(c) \subset \RR^d$ denote the prism of dimension $d$, where $c \geq 0$ is an integer, whose vertices are
\[
\eb_1, \ldots, \eb_{d-1}, \, \wb, \, \eb_1 + (c+1) \eb_d, \ldots, \eb_{d-1} + (c+1) \eb_d, \, \wb + (c+1) \eb_d.
\]    
One has    
\begin{eqnarray}
\label{RK}
b(\Qc(c)) 
= cd + (2d+2), \, \, \, c(\Qc(c)) = c, \, \, \, {\rm vol}(\Qc(c)) = (c+1)d^2/d!.
\end{eqnarray}
It is observed in \cite[Example 3.1]{kawaguchi} that $\Qc(c)$ is Castelnuovo if $c \geq 1$.

Now, imitating the prism $\Qc(c) \subset \RR^d$, one can introduce the lattice polytope $$\Qc_j(c) \subset \RR^d,$$ called {\em of prism type}, of dimension $d$, where $c \geq 0$ and $0 \leq j \leq d-1$ are integers, whose vertices are
\[
\eb_1, \ldots, \eb_{d-1}, \, \wb, \, (c+1) \eb_d, \, \eb_1 + (c+1) \eb_d, \ldots, 
\eb_j + (c+1) \eb_d, 
\]
\[
\eb_{j+1} + c \eb_d, \ldots, \eb_{d-1} + c \eb_d, \, \wb + c \eb_d.
\]    

A routine computation based on the configuration of the lattice points of $\Qc_j(c)$ yields that 
\begin{eqnarray}
\label{Boston}
b(\Qc_j(c)) 
= cd + (d+2+j), \, \, \, \, \, c(\Qc_j(c)) = c
\end{eqnarray}
and
\begin{eqnarray}
\label{Chris}
{\rm vol}(\Qc_j(c)) = {\rm vol}(\Qc(c-1)) + {\rm vol}(\Qc_j(0)).
\end{eqnarray}

\begin{Lemma}
\label{abcdefg}
One has
\[
{\rm vol}(\Qc_j(0)) = (jd + d  - j)/d!
\]
\end{Lemma}

\begin{proof}
Let
\[
V = \{ \eb_1, \ldots, \eb_{d-1}, \, \wb, \, \eb' _0 = \eb_d, \, \eb' _1 = \eb_1 + \eb_d, \ldots, \eb'_j = \eb_j + \eb_d\} 
\]
denote the set of vertices of $\Qc_j(0)$ and ${\bf 0} = (0,\ldots,0) \in \RR^d$.  Let $\Bc$ denote the set of those subsets $W \subset V$ for which
\begin{itemize}
\item[(i)]
$|W| = d$;
\item[(ii)]
$\{\eb' _s, \eb_t\} \nsubseteq W, \, \, \, \, \, 0 \leq s < t \leq j$;
\item[(iii)]
$W \neq \{\eb_1, \ldots, \eb_{d-1}, \, \wb\}$;
\item[(iv)]
$W \neq \{\eb_1, \ldots, \eb_{\xi - 1}, \, \eb' _\xi, \, \eb' _{\xi+1}, \ldots, \eb' _j, \, \eb_{j+1}, \ldots, \eb_{d-1}, \, \wb \}, \, \, \, \, \, 1 \leq \xi \leq j$.
\end{itemize} 
If $W \in \Bc$, then $\sigma_W=\conv(W \cup \{{\bf 0}\}) \subset \RR^d$ is a simplex of dimension $d$ and ${\rm vol}(\sigma_W) = 1/d!$.  Let $\Sigma$ denote the set of simplices consisting of $\sigma_W, \, W \in \Bc$, together with their faces.  A standard technique for triangulations \cite[Chapter 4]{gtm279} guarantees that $\Sigma$ is a triangulation of $\Qc_j(0)$.  Hence ${\rm vol}(\Qc_j(0)) = |\Bc|/d!$. 

Now, one claims $|\Bc| = jd + d - j$.  Let $W \in \Bc$ and $\xi_W = \min\{ \xi : \eb'_\xi \in W\}$.  The number of $W \in \Bc$ with $\xi_W = 0$ is $d$. The number of $W \in \Bc$ with $\xi_W = i$ is $d-1$ for $1 \leq i \leq j$.  Thus $|\Bc| = d + j(d-1)$, as desired.  
\end{proof}

Corollary \ref{XXXXX} follows from Lemma \ref{abcdefg} together with (\ref{Chris}).

\begin{Corollary}
\label{XXXXX}
One has 
\[
{\rm vol}(\Qc_j(c)) = (cd^2 + jd + d - j)/d!.
\]
Hence $\Qc_j(c)$ is Castelnuovo if $c \geq 1$. 
\end{Corollary}

\section{A partial list of the Castelnuovo triplets}
A partial list of the Castelnuovo triplets will be presented.  

\begin{Example}
\label{HT}
Let $c \geq 1$ and $d \geq 3$ be integers.  It is shown \cite{hibi_tsuchiya} that the lattice polytope $\Pc_{c,d} \subset \RR^d$ of dimension $d$ whose vertices are 
\[
{\bf 0}, \, \, \eb_1, \ldots, \eb_{d-1}, \, \, \eb_d + cd \sum_{i=1}^{d} \eb_i
\]
is Castelnuovo whose Castelnuovo triplet is $(d + 1, c, d)$.  
\end{Example}

Now, Theorem \ref{YYYYY} below follows from Theorem \ref{highlight} together with (\ref{RK}), (\ref{Boston}) and Example \ref{HT}. 

\begin{Theorem}
\label{YYYYY}
Let $c \geq 1, \, d \geq 3$ and $d + 1 \leq b \leq cd + (2d + 2)$ be integers.  Then the triplet 
$(b, c, d)$ is Castelnuovo.  
\end{Theorem} 


\begin{Conjecture}
\label{abcdefgh}
Let $c \geq 2, \, d \geq 3$ and $b \geq d + 1$ be integers.  Then the triplet $(b, c, d)$ is Castelnuovo if and only if $d + 1 \leq b \leq cd + (2d + 2)$.  
\end{Conjecture} 

The Conjecture \ref{abcdefgh} is true for $d = 2$ (Scott \cite{Scott}) and is open for $d \geq 3$.

\begin{Example}
Let $d \geq 3$ and $\Cc_d \subset \RR^d$ the standard unit cube which is the convex hull of the $2^d$ points $\pm \eb_1 \pm \eb_2 \pm \cdots \pm \eb_d$ of $\RR^d$.  It follows from a simple computation that $\Cc_d$ is Castelnuovo if and only if $d = 3$.  When $d = 3$, one has
\[
b(\Cc_3) = 26, \, \, \, \, \, c(\Cc_3) = 1, \, \, \, \, \, {\rm vol}(\Cc_3) = 8.
\]
\end{Example}

\begin{Example}
Let $d = 3$ and $\Qc \subset \RR^3$ the lattice polytope with the vertices
\[
(-1,-1,1), \, (2,-1,1), \, (-1,2,1), \, (-1,-1,-1), \, (2,-1,-1), \, (-1,2,-1)
\] 
Then $\Qc$ is Castelnuovo with 
\[
b(\Pc) = 29, \, \, \, \, \, c(\Pc) = 1, \, \, \, \, \, {\rm vol}(\Pc) = 9.
\]
\end{Example}

\end{document}